\documentclass[a4paper,11pt]{article}

\usepackage[utf8]{inputenc}
\usepackage[T1]{fontenc}
\usepackage{lmodern}
\usepackage[ngerman,english]{babel}
\usepackage{amsmath, amsfonts, amssymb, amsthm, mathtools, dsfont, tikz, booktabs, mathrsfs, subcaption}
\usepackage{hyperref}
\usepackage[noabbrev]{cleveref}
\usepackage{aliascnt}

\usetikzlibrary{arrows.meta,calc}

\usepackage[left=2.5cm,right=2.5cm,top=2.5cm,bottom=2.5cm]{geometry}

\allowdisplaybreaks[2]

\makeatletter
\let\@fnsymbol\@alph
\makeatother

\definecolor{OSNAred}{HTML}{AD1034}
\definecolor{OSNAgold}{HTML}{FBB900}
\definecolor{OSNAgray}{HTML}{CFD0D1}
\definecolor{RUBblue}{HTML}{003560}  
\definecolor{RUBgreen}{HTML}{94C11C} 
\definecolor{RUBgray}{HTML}{E6E4E4}  

\theoremstyle{plain}

\newtheorem{thm}{Theorem}[section]

    \newaliascnt{lem}{thm}
\newtheorem{lemma}[lem]{Lemma}
    \aliascntresetthe{lem}

\newtheorem{proposition}[thm]{Proposition}

\newtheorem{rem}[thm]{Remark}

\crefname{thm}{theorem}{theorems}
\crefname{lem}{lemma}{lemmas}

\usepackage{enumitem}

\setlist[enumerate]{label=(\alph*)}

\def\indicator{\mathbf{1}}
\def\ee{\mathbf{e}}
\def\conv{\operatorname{conv}}
\def\D{\operatorname{D}}

\def\Cosmological{C}

\def\PP{\mathbb{P}}
\def\EE{\mathbb{E}}

    \def\VV{\mathbb{V}}
    \def\V{\mathbb{V}}
\def\Cov{\mathrm{Cov}}

\def\RR{\mathbb{R}}
\def\NN{\mathbb{N}}

\title{Central limit theorems for high dimensional lattice polytopes: cosmological polytopes} 
\author{
Torben Donzelmann\thanks{Universit\"at Osnabr\"uck, Germany. \url{torben.donzelmann@uos.de}}
\and
Martina Juhnke\thanks{Universit\"at Osnabr\"uck, Germany. \url{martina.juhnke@uos.de}}
\and
Benedikt Redno\ss\thanks{Ruhr University Bochum, Germany. \url{benedikt.rednoss@rub.de}}
\and
Christoph Th\"ale\thanks{Ruhr University Bochum, Germany. \url{christoph.thaele@rub.de}}
}
\date{}

\usepackage{todonotes}
\reversemarginpar
\begin{document}

\maketitle

\begin{abstract}
\noindent
We study cosmological polytopes induced by Erd\H{o}s--R\'enyi random graphs in a high-dimen\-sional regime. These graph-based lattice polytopes form a natural model of random lattice polytopes in which geometric features are determined by the structure of the underlying random graph. Focusing on the number of polytope edges and on the number of edges in unimodular triangulations, we derive asymptotic formulas for expectations and variances and prove quantitative central limit theorems in the relevant parameter regime. The analysis relies on explicit graph-theoretic descriptions of the corresponding edge sets and on the discrete Malliavin--Stein method for normal approximation. \\[0.5cm]
\end{abstract}

\section{Introduction}

Graph polytopes provide a natural bridge between discrete geometry and graph theory: a combinatorial object is encoded by a lattice polytope whose facial structure reflects structural properties of the underlying graph. Over the past years, several such constructions, including graphical zonotopes, matching polytopes, flow polytopes and root polytopes, have attracted increasing attention, both because of their intrinsic geometric interest and because they connect to topics ranging from commutative algebra to mathematical physics, see \cite{Ferroni} and references therein.  Among the most prominent examples are symmetric edge polytopes and cosmological polytopes. Although their definitions are elementary, they give rise to surprisingly rich geometric and combinatorial behavior.

The present paper is concerned with \emph{cosmological polytopes}. These are lattice polytopes associated with finite graphs, originally introduced in \cite{CosmologicalIntorudction}. Their origin lies in mathematical physics, where they were introduced as geometric objects encoding contributions of Feynman diagrams to the wavefunction of the universe in certain cosmological models \cite{Cosmo-Benincasa}. Subsequently, their structure has been investigated from various perspectives, ranging from face theory~\cite{kuehne2024facescosmologicalpolytopes} and triangulations~\cite{TriangulationCosmological} to Ehrhart theory~\cite{CosmoMatroid,CosmoEhrhart} and the computation of canonical forms via dual volumes~\cite{CosmoDual}.

From the point of view of convex geometry, cosmological polytopes provide a graph-based analogue of classical point-generated polytopes. While in usual polytope models one starts with a set of points in Euclidean space and studies the convex hull they generate, the input here is a finite graph and the polytope is obtained from its combinatorial structure. This suggests a natural probabilistic question: what can be said about the typical geometry and combinatorics of cosmological polytopes when the underlying graph is random? In this sense, cosmological polytopes form a particularly rigid and tractable model within a broader probabilistic theory of \emph{random graph polytopes}. More precisely, we consider cosmological polytopes generated by Erd\H{o}s--R\'enyi random graphs \(G_{n,p}\), and study the resulting random lattice polytopes in the high-dimensional regime where $p=p(n)$ and \(n\to\infty\). Our main focus lies on the number of edges of the random cosmological polytope and on the number of edges appearing in unimodular triangulations. We show that these quantities admit explicit expectation and variance asymptotics and satisfy central limit theorems in all natural parameter regimes.

This article is a companion to our earlier paper \cite{donzelmann2026centrallimittheoremshigh} on random symmetric edge polytopes. There, we investigated the edge structure of the symmetric edge polytope of \(G_{n,p}\) and of its unimodular triangulations. A key feature of that model was that the relevant edge counts depended on rather delicate local configurations in the graph, such as the presence or absence of short cycles, which led to a comparatively intricate variance analysis and to an exceptional regime in which the leading variance term vanished. For cosmological polytopes, the situation turns out to be markedly different. Although the underlying random graph model is the same, the edge structure of the polytope is governed by a simpler combinatorial principle, and the resulting fluctuation theory is correspondingly more easy to access.

We now recall the definition of the objects studied in this paper, more details can be found in Section \ref{sec:preliminaries} below. Let \(G=(V,E)\) be a finite graph, and write \((\ee_a)_{a\in V\cup E}\) for the standard basis of \(\RR^{V\cup E}\). The \emph{cosmological polytope} associated with $G$ is defined by
\[
    \Cosmological_G
    \coloneqq
    \conv\left\{
        \begin{array}{l}
            \phantom{-}\ee_i+\ee_j-\ee_f,\\
            \phantom{-}\ee_i-\ee_j+\ee_f,\\
            -\ee_i+\ee_j+\ee_f
        \end{array}
        \;\middle|\;
        f=\{i,j\}\in E
    \right\}
    \subset \RR^{V\cup E}.
\]
To introduce randomness, we let \(G\) be an Erd\H{o}s--R\'enyi random graph \(G_{n,p}\) on the node set \([n]=\{1,\dots,n\}\), where each arc is retained independently with probability \(p=p(n)\in(0,1)\), and define
\[
    \Cosmological_{n,p}\coloneqq \Cosmological_{G_{n,p}}.
\]
A for us particularly useful feature of cosmological polytopes is that their edge structure admits a very explicit graph-theoretic description. As a consequence, the number of edges \(f_1(\Cosmological_{n,p})\) of \(\Cosmological_{n,p}\)  can be expressed directly in terms of two elementary random variables of the underlying random graph, namely its number of arcs and its number of leaves, that is, nodes of degree one. This makes the model especially accessible to probabilistic analysis.

Our main results show that this quantity exhibits asymptotically Gaussian fluctuations in the high-dimensional regime where $n\to\infty$ and $p=p(n)$ is at least of order $n^{-2}$. In a simplified form, and under natural assumptions ensuring divergence of the variance, we prove that
\[
    \EE[f_1(\Cosmological_{n,p})]\simeq n^4p^2
    \qquad\text{and}\qquad
    \VV(f_1(\Cosmological_{n,p}))\gtrsim n^6p^3(1-p),
\]
and moreover
\[
    \frac{f_1(\Cosmological_{n,p})-\EE[f_1(\Cosmological_{n,p})]}
    {\sqrt{\VV(f_1(\Cosmological_{n,p}))}}
    \xrightarrow{d} \mathcal{N}(0,1).
\]
We also establish an analogous central limit theorem for the number of edges in unimodular triangulations of \(\Cosmological_{n,p}\). The statements given here are deliberately informal and only intended to illustrate the general nature of the results. The precise assumptions, asymptotic regimes, and quantitative estimates are formulated in full detail throughout the paper. In particular, beyond convergence in distribution, we derive explicit rates of convergence for the normal approximation by means of the discrete Malliavin--Stein method. This also uncovers an interesting contrast with the companion paper \cite{donzelmann2026centrallimittheoremshigh} on symmetric edge polytopes. There, the number of edges in a unimodular triangulation is combinatorially more accessible than the number of polytope edges itself. For cosmological polytopes, the situation is essentially reversed: the number of polytope edges admits a particularly simple representation in terms of basic graph statistics, whereas the corresponding quantity for unimodular triangulations is slightly more involved. Still, both models fit naturally into the same general probabilistic framework.

From a broader perspective, the present work together with the companion paper \cite{donzelmann2026centrallimittheoremshigh} contributes to a probabilistic theory of random graph polytopes. These models differ fundamentally from classical random polytopes generated by random points in Euclidean space. Their geometry is not produced by a cloud of independent points but is encoded by a random discrete structure. This makes them particularly well suited for combining tools from random graph theory, lattice polytope theory, and probability. To the best of our knowledge, the present paper and \cite{donzelmann2026centrallimittheoremshigh} provide the first systematic examples of probabilistic results for high-dimensional random lattice polytopes. 

\smallskip

The remainder of the paper is organized as follows. In \Cref{sec:preliminaries} we collect the necessary background on cosmological polytopes, their triangulations, Erd\H{o}s--R\'enyi random graphs and the normal approximation bound we work with. In \Cref{sec:CosmoEdges} we formulate and develop our main results for the edge count of random cosmological polytopes, whereas Section \ref{sec:CosmoTriangulations} contains our main results for the edge count of their unimodular triangulations.

\section{Preliminaries}\label{sec:preliminaries}
The goal of this section is to recall some background from probability theory and the theory of polytopes.  To avoid text overlap and since the present article is the second one in a sequence of two articles, we decided to keep this section short and to only mention non-standard notions and tools. For further material, we refer the reader to  \cite[Section 2]{donzelmann2026centrallimittheoremshigh}.

\subsection{General notation and random graph notation}

 For two sequences \((a_n)_{n\in\NN}\) and \((b_n)_{n\in\NN}\) of real numbers, we write
\[
\begin{aligned}
a_n \lesssim b_n,
&\qquad\text{if}\qquad
\limsup_{n\to\infty}\left|\frac{a_n}{b_n}\right|<\infty,\\[0.3em]
a_n \simeq b_n,
&\qquad\text{if}\qquad
0<
\liminf_{n\to\infty}\left|\frac{a_n}{b_n}\right|
\leq
\limsup_{n\to\infty}\left|\frac{a_n}{b_n}\right|
<\infty,\\[0.3em]
a_n \gg b_n,
&\qquad\text{if}\qquad
\lim_{n\to\infty}\left|\frac{a_n}{b_n}\right|=\infty.
\end{aligned}
\]
Moreover, for \(m\in\NN\) and \(p\in(0,1)\), we write
$X\sim \mathrm{Bin}(m,p)$ if \(X\) is a binomial random variable with parameters \(m\) and \(p\).

Let \(G=(V,E)\) be a finite graph. Throughout this paper, we refer to the elements of \(V\) as \emph{nodes} and to the elements of \(E\) as \emph{arcs}, in order to distinguish them from the vertices and edges of the associated cosmological polytope introduced in the next subsection. For a node \(v\in V\), we write \(\deg(v)\) for its degree. A node of degree \(1\) is called a \emph{leaf}, and a node of degree at least \(1\) is called \emph{non-isolated}. For \(n\in\NN\) and \(p=p(n)\in(0,1)\), let \(G_{n,p}\) denote the Erd\H{o}s--R\'enyi random graph on node set \([n]:=\{1,\ldots,n\}\), in which each arc is present independently with probability \(p\). We write
\[
E_{n,p}\coloneqq |E(G_{n,p})|
\qquad\text{and}\qquad
L_{n,p}\coloneqq |\{v\in[n] ~|~\deg(v)=1\}|
\]
for the number of arcs and the number of leaves of \(G_{n,p}\), respectively.

\subsection{Cosmological polytopes and their triangulations}\label{sec:prelim_cosmo}

We now recall the deterministic structural results on cosmological polytopes and their triangulations that will be used throughout the paper. Let \(G=(V,E)\) be a finite graph, and let \((\ee_a)_{a\in V\cup E}\) denote the standard basis of \(\RR^{V\cup E}\), indexed by both the nodes and the arcs of \(G\). The associated \emph{cosmological polytope} is defined by
\[
    \Cosmological_G
    \coloneqq
    \conv\left\{
        \begin{array}{l}
            \phantom{-}\ee_i+\ee_j-\ee_f,\\
            \phantom{-}\ee_i-\ee_j+\ee_f,\\
            -\ee_i+\ee_j+\ee_f
        \end{array}
        \;\middle|\;
        f=\{i,j\}\in E
    \right\}
    \subset \RR^{V\cup E}.
\]
Thus, each arc \(f=\{i,j\}\in E\) gives rise to three vertices of \(\Cosmological_G\), which we denote by
\[
t_f\coloneqq \ee_i+\ee_j-\ee_f, \qquad
y_{ij}\coloneqq \ee_i-\ee_j+\ee_f, \qquad
y_{ji}\coloneqq -\ee_i+\ee_j+\ee_f.
\]
%
%
In particular, \(\Cosmological_G\) is a lattice polytope contained in the affine subspace
\[
\left\{
x\in\RR^{V\cup E}:
x_v=0 \ \forall v\in I,\ 
\sum_{v\in V\setminus I} x_v+\sum_{e\in E} x_e =1
\right\}\subset\RR^{V\cup E},
\]
where \(I \subset V\) denotes the set of isolated nodes of \(G\). Consequently, its dimension is $|V|+|E|-1-|I|$. We note that this definition also makes sense for graphs with loops or multiple arcs, although throughout this paper we will only deal with simple graphs.

The combinatorial structure of \(\Cosmological_G\) is governed by an explicit graph-theoretic characterization of its edges, proved in \cite[Corollary 3.2]{kuehne2024facescosmologicalpolytopes}.

\begin{proposition}\label{prop:EdgeCharCosmo}
Let \(G=(V,E)\) be a graph. Two distinct vertices \(\mathbf{x},\mathbf{y}\) of \(\Cosmological_G\) form an edge of \(\Cosmological_G\) if and only if one of the following holds:
\begin{enumerate}
    \item \(\mathbf{x}\) and \(\mathbf{y}\) correspond to different arcs of \(G\);
    \item \(\mathbf{x}\) and \(\mathbf{y}\) correspond to the same arc \(f=\{i,j\}\in E\), and either
    \begin{enumerate}
        \item \(\{\mathbf{x},\mathbf{y}\}=\{y_{ij},y_{ji}\}\), or
        \item \(\{\mathbf{x},\mathbf{y}\}=\{y_{ij},t_f\}\) and \(i\) is a leaf of \(G\).
    \end{enumerate}
\end{enumerate}
\end{proposition}

In particular, this shows that the number of edges of \(\Cosmological_G\) is determined by the number of arcs and the number of leaves of the underlying graph $G$. We shall exploit this reduction in the probabilistic analysis below.

We will also investigate the number of edges in unimodular triangulations of cosmological polytopes. Recall that a \emph{unimodular triangulation} of a lattice polytope is a triangulation into lattice simplices of minimal possible volume. By \cite[Corollary 2.10]{TriangulationCosmological}, every cosmological polytope admits such a triangulation. It is easy to see that the only lattice points of $\Cosmological_G$ are its vertices $t_f$, $y_{ij}$, $y_{ji}$ for every arc $f=\{i,j\}$ of $G$ and the standard unit vectors $\ee_f$ and $\ee_i$ for every arc $f$ of $G$ and every non-isolated node $i$ of $G$. Any unimodular triangulation of $\Cosmological_G$ necessarily needs to use all these lattice points as vertices. 
 \cite[Theorem 2.9]{TriangulationCosmological} provides an explicit description of the edges appearing in a suitable unimodular triangulation.

\begin{proposition}\label{prop:EdgeCharCosmoTriang}
Let \(G=(V,E)\) be a graph. Then there exists a unimodular triangulation \(\mathcal{T}\) of \(\Cosmological_G\) such that two distinct vertices \(\mathbf{x},\mathbf{y}\) of \(\mathcal{T}\) form an edge of \(\mathcal{T}\) if and only if one of the following holds:
\begin{enumerate}
    \item \(\mathbf{x}\) and \(\mathbf{y}\) correspond to different arcs of \(G\);
    \item \(\{\mathbf{x},\mathbf{y}\}=\{\ee_f,y_{ij}\}\) for some arc \(f=\{i,j\}\in E\);
    \item \(\mathbf{x}=\ee_i\) for some non-isolated node \(i\in V\), and \(\mathbf{y}\neq y_{ij}\) for every arc \(\{i,j\}\in E\).
\end{enumerate}
\end{proposition}
Similar to the edge count of $\Cosmological_G$, the previous proposition shows that the edge count of any unimodular triangulation of $\Cosmological_G$ only depends on the number of arcs and the number of isolated nodes of the underlying graph. This will be essential  for the proofs of our main results. In the sections below, we will  express the relevant edge counts in terms of simple graph statistics of \(G_{n,p}\), thereby reducing the geometric problem to a probabilistic one.

\subsection{Discrete gradients and a normal approximation bound}\label{sec:prelim_prob}

For the proof of the central limit theorems, we use a standard Rademacher encoding of the Erd\H{o}s--R\'enyi graph and a normal approximation bound on product spaces. Let \(n\geq 1\), let \(p_1,\dots,p_n\in(0,1)\), and set \(q_i\coloneqq 1-p_i\) for \(i\in[n]\). Consider independent random variables \(X_1,\dots,X_n\) with
\[
\PP(X_i=1)=p_i,
\qquad
\PP(X_i=-1)=q_i.
\]
If \(F=F(X_1,\dots,X_n)\) is a measurable real-valued function of $X_1,\ldots,X_n$, we write \(F_i^+\) and \(F_i^-\) for the random variables obtained from \(F\) by setting the \(i\)\textsuperscript{th} coordinate of \((X_1,\dots,X_n)\) to \(+1\) and \(-1\), respectively. The corresponding discrete gradient is defined by
\[
\D_iF\coloneqq \sqrt{p_iq_i}\,(F_i^+-F_i^-).
\]
For \(i,j\in[n]\), the second-order discrete gradient \(\D_i\D_jF\) is defined by applying \(\D_i\) to \(\D_jF\). Explicitly,
\[
\D_i\D_jF
=
\sqrt{p_iq_ip_jq_j}\,
\bigl(
F_{j,i}^{+,+}-F_{j,i}^{+,-}-F_{j,i}^{-,+}+F_{j,i}^{-,-}
\bigr),
\]
where \(F_{j,i}^{\alpha,\beta}\) denotes the random variable obtained by fixing the \(j\)\textsuperscript{th} and $i$\textsuperscript{th} coordinate at \(\{\pm1\}\), where $\alpha$ and $\beta$ determine the sign, respectively. 
We note that \(\D_i\D_iF=0\).

Throughout the paper, distributional convergence to the Gaussian law will be measured in Kolmogorov distance,
\[
d_K(Y,Z)\coloneqq \sup_{z\in\RR}\bigl|\PP(Y\le z)-\PP(Z\le z)\bigr|,
\]
for real-valued random variables \(Y\) and \(Z\). The following estimate is our main tool. It is a finite-dimensional specialization of \cite[Theorem 4.1]{zgstheorem}.

\begin{proposition}\label{prop:2ndOrderPoincare}
Let \(F=F(X_1,\dots,X_n)\) be a measurable real-valued function of independent Rademacher random variables as above, and assume that \(\EE[F]=0\) and \(\EE[F^2]=1\). Define
\begin{align*}
    B_1 &\coloneqq \sum_{j,k,\ell=1}^n \sqrt{\EE[(\D_jF)^2(\D_kF)^2]}\,
    \sqrt{\EE[(\D_\ell\D_jF)^2(\D_\ell\D_kF)^2]},\\
    B_2 &\coloneqq \sum_{j,k,\ell=1}^n \frac{1}{p_\ell q_\ell}\,
    \EE[(\D_\ell\D_jF)^2(\D_\ell\D_kF)^2],\\
    B_3 &\coloneqq \sum_{k=1}^n \frac{1}{p_kq_k}\,\EE[(\D_kF)^4],\\
    B_4 &\coloneqq \sum_{k,\ell=1}^n \frac{1}{p_kq_k}\,
    \sqrt{\EE[(\D_kF)^4]}\,\sqrt{\EE[(\D_\ell\D_kF)^4]},\\
    B_5 &\coloneqq \sum_{k,\ell=1}^n \frac{1}{p_kq_kp_\ell q_\ell}\,
    \EE[(\D_\ell\D_kF)^4].
\end{align*}
Then, for a standard Gaussian random variable \(N\sim\mathcal N(0,1)\),
\[
d_K(F,N)
\le
\frac{\sqrt{15}}{2}\sqrt{B_1}
+\frac{\sqrt{3}}{2}\sqrt{B_2}
+2\sqrt{B_3}
+2\sqrt{6}\sqrt{B_4}
+2\sqrt{3}\sqrt{B_5}.
\]
\end{proposition}

In our applications, the coordinates \(X_i\) will encode the presence or absence of the arcs of \(G_{n,p}\). The relevant edge counts can then be viewed as functionals of independent Rademacher variables, and the proof of the quantitative central limit theorems reduces to estimating their first- and second-order discrete gradients. Once suitable bounds are available, \Cref{prop:2ndOrderPoincare} yields both asymptotic normality and an explicit rate of convergence.

\section{Edges of random cosmological polytopes}\label{sec:CosmoEdges}

In this section, we study the number of edges of the random cosmological polytope \(\Cosmological_{n,p}\) generated by an Erd\H{o}s--R\'enyi random graph \(G_{n,p}\). As explained in \Cref{sec:prelim_cosmo}, the edge structure of cosmological polytopes is particularly transparent: the number of polytope edges can be expressed directly in terms of two basic graph statistics of the underlying graph, namely the number of arcs and the number of leaves. This makes the analysis considerably more direct than in the symmetric edge polytope setting studied in \cite{donzelmann2026centrallimittheoremshigh}.

Throughout this section, \(G_{n,p}\) denotes the Erd\H{o}s--R\'enyi random graph on vertex set \([n]\) with edge probability \(p=p(n)\in(0,1)\), and we write $E_{n,p}$ and $L_{n,p}$ for the number of arcs and leaves of \(G_{n,p}\), respectively. Moreover, we let $K_{n,p}\coloneqq f_1(\Cosmological_{n,p})$ denote the number of edges of the random cosmological polytope \(\Cosmological_{n,p}\).

Our goal is to determine the asymptotic behaviour of \(K_{n,p}\) as \(n\to\infty\), including its expectation, variance, and Gaussian fluctuations. The main result of this section can be summarized as follows.

\begin{thm}\label{thm:CosmoEdgesMain}
Let \(K_{n,p}=f_1(\Cosmological_{n,p})\) be the number of edges of the random cosmological polytope associated with \(G_{n,p}\).
If \(p\gg n^{-2}\), then

\begin{enumerate}
    \item\label{thm:CosmoEdgesMain:exp}
    \hfill\(
        \EE[K_{n,p}] \simeq n^4p^2
        \qquad\text{and}\qquad
        \VV(K_{n,p}) \simeq n^6p^3(1-p).
    \)\hfill\mbox{}

    \item\label{thm:CosmoEdgesMain:clt}
    \hfill\(\displaystyle
        d_K\!\left(
        \frac{K_{n,p}-\EE[K_{n,p}]}{\sqrt{\VV(K_{n,p})}},
        N\right)
        \lesssim \frac{1}{n\sqrt{p(1-p)}},
    \)\hfill\mbox{}\\[0.4\baselineskip]
    where \(N\sim\mathcal N(0,1)\) is a standard Gaussian random variable.
\end{enumerate}
\end{thm}

The proof will be divided into four parts. We first compute the expectation of \(K_{n,p}\), then determine the order of its variance, and afterwards derive suitable bounds on the first- and second-order discrete gradients. These estimates will finally be combined with the normal approximation bound from \Cref{prop:2ndOrderPoincare} to prove the quantitative central limit theorem.

The starting point is the following deterministic identity we recall from \cite[Theorem 5.1]{kuehne2024facescosmologicalpolytopes}.

\begin{lemma}\label{lem:CosmoEdgeFormula}
For every finite graph \(G=(V,E)\),
\[
f_1(\Cosmological_G)
=
9\binom{|E|}{2}+|E|+ |\{v\in V:\deg(v)=1\}|.
\]
\end{lemma}

Applied to \(G_{n,p}\), this shows that
\begin{equation}\label{eq:CosmoKRepresentation}
K_{n,p}
=
9\binom{E_{n,p}}{2}+E_{n,p}+L_{n,p}.
\end{equation}
Thus, the probabilistic analysis of the polytope edge count reduces to the study of the random graph statistics \(E_{n,p}\) and \(L_{n,p}\).

\subsection{Expectation}\label{sec:CosmoExpectation}

We begin with the expectation of \(K_{n,p}\). Since \eqref{eq:CosmoKRepresentation} expresses \(K_{n,p}\) as a simple combination of \(E_{n,p}\) and \(L_{n,p}\), the computation is immediate once the expectations of these quantities are known.

\begin{lemma}\label{lem:CosmoExpectationExact}
Let \(K_{n,p}=f_1(\Cosmological_{n,p})\). Then
\[
\EE[K_{n,p}]
=
9\binom{\binom{n}{2}}{2}p^2
+
\binom{n}{2}p
+
n(n-1)p(1-p)^{n-2}.
\]
\end{lemma}
\begin{proof}
By \eqref{eq:CosmoKRepresentation},
\[
\EE[K_{n,p}]
=
9\,\EE\!\left[\binom{E_{n,p}}{2}\right]
+\EE[E_{n,p}]
+\EE[L_{n,p}].
\]
Since \(E_{n,p}\sim \mathrm{Bin}\!\left(\binom{n}{2},p\right)\), we have
\[
\EE[E_{n,p}] = \binom{n}{2}p\qquad\text{and}\qquad\EE\!\left[\binom{E_{n,p}}{2}\right]
=
\binom{\binom{n}{2}}{2}p^2.
\]
Indeed, \(\binom{E_{n,p}}{2}\) counts unordered pairs of present arcs, and each such pair occurs with probability \(p^2\).
It remains to compute \(\EE[L_{n,p}]\). Writing \(L_{n,p}=\sum_{v\in[n]}{\bf 1}_{\{\deg(v)=1\}}\), we obtain
\[
\EE[L_{n,p}]
=
\sum_{v\in[n]} \PP(\deg(v)=1).
\]
For a fixed node \(v\), there are \(n-1\) choices for its unique neighbor, and all remaining \(n-2\) arcs incident to \(v\) are forced to be absent. Hence $\PP(\deg(v)=1)=(n-1)p(1-p)^{n-2}$,
so that $\EE[L_{n,p}]=n(n-1)p(1-p)^{n-2}$.
Combining these identities yields the claim.
\end{proof}

The exact formula in \Cref{lem:CosmoExpectationExact} immediately implies the generic expectation asymptotics in the non-degenerate regime.

\begin{proof}[Proof of \Cref{thm:CosmoEdgesMain}~\ref{thm:CosmoEdgesMain:exp}, expectation.]
By \Cref{lem:CosmoExpectationExact},
\[
\EE[K_{n,p}]
=
9\binom{\binom{n}{2}}{2}p^2
+
\binom{n}{2}p
+
n(n-1)p(1-p)^{n-2}.
\]
The first term is of order \(n^4p^2\), the second is of order \(n^2p\), and the third is at most of order \(n^2p\). Since \(p\gg n^{-2}\), we have \(n^2p \ll n^4p^2\), and therefore the first term dominates. This proves the claim.
\end{proof}

\subsection{Variance}\label{sec:CosmoVariance}

We next determine the order of the variance of \(K_{n,p}\), thereby completing the proof of \Cref{thm:CosmoEdgesMain}~\ref{thm:CosmoEdgesMain:exp}. As in the computation for the expectation, the key point is that the edge count of \(\Cosmological_{n,p}\) can be written in terms of the number of arcs and the number of leaves of \(G_{n,p}\). The dominant contribution comes from the quadratic term in \(E_{n,p}\), while the leaf term is of lower order.

\begin{lemma}\label{lem:CosmoVariance}
Let \(K_{n,p}=f_1(\Cosmological_{n,p})\). Then
\[
\VV(K_{n,p}) \begin{cases}
    \simeq n^6p^3(1-p),\qquad&\text{if } p\gg n^{-2}\\
    \lesssim n^2p, \qquad & \text{if } p\lesssim n^{-2}.
\end{cases}
\]
\end{lemma}

\begin{proof}
Recall from \eqref{eq:CosmoKRepresentation} that
\[
K_{n,p}=9\binom{E_{n,p}}{2}+E_{n,p}+L_{n,p}.
\]
Hence
\begin{align}
\VV(K_{n,p})
&=
81\,\VV\!\left(\binom{E_{n,p}}{2}\right)
+\VV(E_{n,p})
+\VV(L_{n,p}) \notag\\
&\quad
+18\,\Cov\!\left(\binom{E_{n,p}}{2},E_{n,p}\right)
+18\,\Cov\!\left(\binom{E_{n,p}}{2},L_{n,p}\right)
+2\,\Cov(E_{n,p},L_{n,p}).
\label{eq:CosmoVarianceDecomposition}
\end{align}

We first analyze the dominant term. Since \(E_{n,p}\sim \mathrm{Bin}\!\left(\binom{n}{2},p\right)\), the random variable \(\binom{E_{n,p}}{2}\) counts unordered pairs of present arcs. Writing it as a sum of indicator random variables over all unordered pairs of distinct arcs, we obtain
\[
\binom{E_{n,p}}{2}
=
\sum_{\{e,f\} \in \binom{[n]}{2}}
{\bf 1}_{\{ \{e,f\}\subseteq E(G_{n,p})\}} .
\]
Two such indicators are independent unless the corresponding pairs of arcs share at least one arc. Consequently,
\begin{align*}
\VV\!\left(\binom{E_{n,p}}{2}\right)
&=
3\binom{\binom{n}{2}}{3}\bigl(p^3-p^4\bigr)
+
\binom{\binom{n}{2}}{2}\bigl(p^2-p^4\bigr) \\
&=
3\binom{\binom{n}{2}}{3}p^3(1-p)
+
\binom{\binom{n}{2}}{2}p^2(1-p^2).
\end{align*}
In particular,
\begin{equation}\label{eq:VarChooseEnp}
\VV\!\left(\binom{E_{n,p}}{2}\right)
\simeq n^6p^3(1-p)+n^4p^2(1-p^2).
\end{equation}
If \(p\gg n^{-2}\), then \(n^4p^2  \ll n^6p^3 \), and hence
\begin{equation}\label{eq:VarChooseEnpAsymptotic}
\VV\!\left(\binom{E_{n,p}}{2}\right)\simeq n^6p^3(1-p).
\end{equation}

Next,
\[
\VV(E_{n,p})=\binom{n}{2}p(1-p)\lesssim n^2p(1-p),
\]
and it remains to bound \(\VV(L_{n,p})\). Write
\[
L_{n,p}=\sum_{v\in[n]} {\bf 1}_{\{\deg(v)=1\}}.
\]
For a fixed node \(v\in[n]\),
\[
\PP(\deg(v)=1)=(n-1)p(1-p)^{n-2},
\]
and therefore
\[
\VV\bigl({\bf 1}_{\{\deg(v)=1\}}\bigr)
\leq (n-1)p(1-p)^{n-2}.
\]
Hence
\begin{equation}\label{eq:VarLeafsDiagSharp}
\sum_{v\in[n]}\VV\bigl({\bf 1}_{\{\deg(v)=1\}}\bigr)
\lesssim n^2p(1-p)^{n-2}.
\end{equation}

Distinct nodes  \(v,w\in[n]\) are leaves in exactly two situations: either the arc \(vw\) is present and all other arcs incident to \(v\) or \(w\) are absent, or each of \(v\) and \(w\) is connected to a unique neighbour different from the other. In the first case, \(2n-4\) additional arcs must be absent, and in the second case, \(2n-5\) additional arcs must be absent. Thus
\[
\PP\bigl(\deg(v)=1,\deg(w)=1\bigr)
=
p(1-p)^{2n-4}
+
(n-2)^2p^2(1-p)^{2n-5}.
\]
It follows that
\begin{align*}
\Cov\bigl({\bf 1}_{\{\deg(v)=1\}},{\bf 1}_{\{\deg(w)=1\}}\bigr)
&=
p(1-p)^{2n-4}
+(n-2)^2p^2(1-p)^{2n-5} -(n-1)^2p^2(1-p)^{2n-4} \\
&=
p(1-p)^{2n-5}\bigl(1-(n-1)p\bigr)^2.
\end{align*}
In particular, the covariance is always non-negative. Since
\[
\bigl(1-(n-1)p\bigr)^2 \leq 2+2(n-1)^2p^2 \lesssim 1+n^2p^2,
\]
we obtain
\[
\Cov\bigl({\bf 1}_{\{\deg(v)=1\}},{\bf 1}_{\{\deg(w)=1\}}\bigr)
\lesssim p(1-p)^{2n-5}+n^2p^3(1-p)^{2n-5},
\]
and hence
\begin{equation}\label{eq:VarLeafsOffDiagSharp}
\sum_{v\neq w}
\Cov\bigl({\bf 1}_{\{\deg(v)=1\}},{\bf 1}_{\{\deg(w)=1\}}\bigr)
\lesssim n^2p(1-p)^{2n-5}+n^4p^3(1-p)^{2n-5}.
\end{equation}
Combining \eqref{eq:VarLeafsDiagSharp} and \eqref{eq:VarLeafsOffDiagSharp}, we arrive at
\begin{equation}\label{eq:VarLeafsBound}
\VV(L_{n,p})
\lesssim
n^2p(1-p)^{n-2}
+
n^2p(1-p)^{2n-5}
+
n^4p^3(1-p)^{2n-5}.
\end{equation}

We now return to \eqref{eq:CosmoVarianceDecomposition}. By the Cauchy--Schwarz inequality,
\[
|\Cov(X,Y)|\leq \sqrt{\VV(X)\VV(Y)}
\]
for any square-integrable random variables \(X\) and \(Y\). Using \eqref{eq:VarChooseEnpAsymptotic}, \eqref{eq:VarLeafsBound}, and the bound for \(\VV(E_{n,p})\), we see that
\[
\VV(E_{n,p})\ll n^6p^3(1-p),
\]
if \(p\gg n^{-2}\). 
Moreover, each term in \eqref{eq:VarLeafsBound} is also of smaller order than \(n^6p^3(1-p)\). Indeed,
\[
\frac{n^2p(1-p)^{n-2}}{n^6p^3(1-p)}
=
\frac{(1-p)^{n-3}}{n^4p^2}
\leq \frac{1}{n^4p^2}\to 0,
\]
\[
\frac{n^2p(1-p)^{2n-5}}{n^6p^3(1-p)}
=
\frac{(1-p)^{2n-6}}{n^4p^2}
\leq \frac{1}{n^4p^2}\to 0,
\]
and
\[
\frac{n^4p^3(1-p)^{2n-5}}{n^6p^3(1-p)}
=
\frac{(1-p)^{2n-6}}{n^2}
\leq \frac{1}{n^2}\to 0.
\]
Hence,
\[
\VV(L_{n,p}) \ll n^6p^3(1-p).
\]
By Cauchy--Schwarz, the same is then true for all covariance terms in \eqref{eq:CosmoVarianceDecomposition}. Since the first term on the right-hand side of \eqref{eq:CosmoVarianceDecomposition} is positive and asymptotic to \(81\,n^6p^3(1-p)\), we conclude that
\[
\VV(K_{n,p})\simeq n^6p^3(1-p),
\qquad\text{if } p\gg n^{-2}.
\]

Finally, if \(p\lesssim n^{-2}\), then \(n^4p^2\lesssim n^2p\), \(n^6p^3\lesssim n^2p\), and \(n^4p^3\lesssim n^2p\). Thus \eqref{eq:VarChooseEnp} and \eqref{eq:VarLeafsBound} imply
\[
\VV\!\left(\binom{E_{n,p}}{2}\right)\lesssim n^2p,
\qquad
\VV(E_{n,p})\lesssim n^2p,
\qquad
\VV(L_{n,p})\lesssim n^2p.
\]
Using once more the Cauchy--Schwarz inequality in \eqref{eq:CosmoVarianceDecomposition}, we obtain
\[
\VV(K_{n,p})\lesssim n^2p.
\]
This completes the proof.
\end{proof}

\begin{rem}
The proof shows that the variance behaviour of \(K_{n,p}\) is much simpler than in the symmetric edge polytope setting of \cite{donzelmann2026centrallimittheoremshigh}. There, the leading covariance contributions partially cancel in a distinguished parameter regime, whereas for cosmological polytopes the dominant term is already visible in the variance of \(\binom{E_{n,p}}{2}\).
\end{rem}

\subsection{Discrete gradients and a quantitative central limit theorem}\label{sec:CosmoCLT}

We now prepare the proof of the normal approximation bound by estimating the discrete gradients of \(K_{n,p}\). Since \(K_{n,p}\) is given by the simple representation \eqref{eq:CosmoKRepresentation},
its dependence on a single arc variable can be analyzed rather explicitly. As in \Cref{sec:prelim_prob}, we encode the random graph \(G_{n,p}\) by independent Rademacher variables indexed by the possible arcs \(e\in \binom{[n]}{2}\). For such an arc \(e\), the operator \(\D_e\) denotes the corresponding discrete gradient.

\begin{lemma}\label{lem:CosmoFirstGradient}
For every \(e\in \binom{[n]}{2}\),
\[
\EE\bigl[(\D_e K_{n,p})^4\bigr]\lesssim n^8p^6(1-p)^2.
\]
\end{lemma}

\begin{proof}
Fix \(e=uv\in \binom{[n]}{2}\). Since \(K_{n,p}=9\binom{E_{n,p}}{2}+E_{n,p}+L_{n,p}\), we can treat the three terms separately.

Changing the state of the arc \(e\) changes \(E_{n,p}\) by exactly one. Hence
\[
\D_e E_{n,p}=\sqrt{p(1-p)}.
\]
Moreover, if \(e\) is switched on or off, then the quantity \(\binom{E_{n,p}}{2}\) changes by the current number of present arcs not equal to \(e\). In particular,
\begin{equation}\label{eq:discreteGradientForSets}
\left|\D_e \binom{E_{n,p}}{2}\right|
\leq \sqrt{p(1-p)}\,E_{n,p}.
\end{equation}
Finally, since changing the state of \(e\) can only affect the leaf status of its two endpoints \(u\) and \(v\), we get 
\[
|\D_e L_{n,p}|\leq 2\sqrt{p(1-p)}.
\]
Combining these bounds, we obtain
\[
|\D_e K_{n,p}|
\leq \sqrt{p(1-p)}\,\bigl(9E_{n,p}+3\bigr),
\]
and hence
\[
(\D_e K_{n,p})^4
\lesssim p^2(1-p)^2\bigl(E_{n,p}^4+1\bigr).
\]
Since \(E_{n,p}\sim \mathrm{Bin}\!\left(\binom{n}{2},p\right)\), its fourth moment is of order at most \(n^8p^4\). Consequently,
\[
\EE\bigl[(\D_e K_{n,p})^4\bigr]
\lesssim p^2(1-p)^2\, n^8p^4
=
n^8p^6(1-p)^2,
\]
which proves the claim.
\end{proof}

The second-order gradients are even simpler.

\begin{lemma}\label{lem:CosmoSecondGradient}
For all \(e,f\in \binom{[n]}{2}\),
\[
|\D_f\D_e K_{n,p}|\lesssim p(1-p).
\]
In particular,
\[
\EE\bigl[(\D_f\D_e K_{n,p})^4\bigr]\lesssim p^4(1-p)^4.
\]
\end{lemma}
\begin{proof}
Fix \(e,f\in \binom{[n]}{2}\). Since \(\D_e E_{n,p}=\sqrt{p(1-p)}\), we have \(\D_f\D_e E_{n,p}=0\). On the other hand, the discussion before \eqref{eq:discreteGradientForSets} implies 
\[
\left|\D_f\D_e \binom{E_{n,p}}{2}\right|
\leq p(1-p).
\]

For the leaf term, changing the status of the arc \(e=uv\) can only affect the leaf indicators of the endpoints \(u\) and \(v\). Applying \(\D_f\) once more can therefore only produce a non-zero contribution through the leaf status of at most two nodes, and each such contribution is bounded by \(p(1-p)\). Hence
\[
|\D_f\D_e L_{n,p}|\lesssim p(1-p).
\]
Combining the three terms in the representation of \(K_{n,p}\) yields
\[
|\D_f\D_e K_{n,p}|\lesssim p(1-p),
\]
as claimed.
\end{proof}

We can now prove the normal approximation bound.

\begin{proof}[Proof of Theorem \ref{thm:CosmoEdgesMain} (2)]
Set
\[
F_{n,p}\coloneqq \frac{K_{n,p}-\EE[K_{n,p}]}{\sqrt{\V(K_{n,p})}}.
\]
We apply \Cref{prop:2ndOrderPoincare} to \(F_{n,p}\). By \Cref{lem:CosmoVariance}, we have
\[
\V(K_{n,p})\simeq n^6p^3(1-p),
\qquad\text{if } p\gg n^{-2}.
\]
Moreover, \(\binom{[n]}{2}\) contains of order  \(n^2\) possible arcs.

Using \Cref{lem:CosmoFirstGradient,lem:CosmoSecondGradient}, and Hölder's inequality, we obtain uniformly in the arc indices
\[
\EE\bigl[(\D_e F_{n,p})^4\bigr]
\lesssim
\frac{n^8p^6(1-p)^2}{\V(K_{n,p})^2},
\qquad
\EE\bigl[(\D_f\D_e F_{n,p})^4\bigr]
\lesssim
\frac{p^4(1-p)^4}{\V(K_{n,p})^2}.
\]
Substituting these estimates into the quantities \(B_1,\dots,B_5\) from \Cref{prop:2ndOrderPoincare}, and summing over at most of order \(n^2\), \(n^4\), or \(n^6\) index tuples gives
\[
\begin{aligned}
B_1 &\lesssim \frac{n^{10}p^5(1-p)^3}{\V(K_{n,p})^2},
&\qquad
B_2 &\lesssim \frac{n^6p^3(1-p)^3}{\V(K_{n,p})^2},
&\qquad
B_3 &\lesssim \frac{n^{10}p^5(1-p)}{\V(K_{n,p})^2},\\
B_4 &\lesssim \frac{n^8p^4(1-p)^2}{\V(K_{n,p})^2},
&\qquad
B_5 &\lesssim \frac{n^4p^2(1-p)^2}{\V(K_{n,p})^2},&
\end{aligned}
\]
and hence
\[
B_1\lesssim \frac{1-p}{n^2p},
\qquad
B_2\lesssim \frac{1-p}{n^6p^3},
\qquad
B_3\lesssim \frac{1}{n^2p(1-p)},
\qquad
B_4\lesssim \frac{1}{n^4p^2},
\qquad
B_5\lesssim \frac{1}{n^8p^4}.
\]
In particular, \(B_3\) is the dominant term. Applying \Cref{prop:2ndOrderPoincare}, we conclude that
\[
d_K(F_{n,p},N)\lesssim \sqrt{B_3}\lesssim \frac{1}{n\sqrt{p(1-p)}},
\]
which is the desired estimate.
\end{proof}

\begin{rem}
Compared with the corresponding argument for symmetric edge polytopes in \cite{donzelmann2026centrallimittheoremshigh}, the proof is noticeably shorter. The reason is that the edge count of \(\Cosmological_{n,p}\) depends only on the basic graph statistics \(E_{n,p}\) and \(L_{n,p}\), so that both the first- and second-order discrete gradients can be controlled directly.
\end{rem}

\section{Edges in unimodular triangulations of random cosmological polytopes}
\label{sec:CosmoTriangulations}

We now turn from the edge set of the cosmological polytope itself to the edge set of a unimodular triangulation. Let \(G_{n,p}\) be the Erd\H{o}s--R\'enyi random graph on vertex set \([n]\) with arc probability \(p=p(n)\in(0,1)\), and let \(\Delta_{n,p}\) be a unimodular triangulation of the random cosmological polytope \(\Cosmological_{n,p}\). Throughout this section, we write
\[
    K_{n,p}\coloneqq f_1(\Delta_{n,p})
\]
for the number of edges of this triangulation. Moreover, \(E_{n,p}\) denotes, as before, the number of arcs of \(G_{n,p}\), and we write $\widetilde N_{n,p}\coloneqq|\{v\in[n]:\deg(v)\geq 1\}|$
for the number of non-isolated nodes of \(G_{n,p}\).

\begin{thm}\label{thm:CosmoTriangMain}
Let \(K_{n,p}=f_1(\Delta_{n,p})\) be the number of edges of a unimodular triangulation of the random cosmological polytope \(\Cosmological_{n,p}\). If \(p\gg n^{-2}\), then
\begin{enumerate}
    \item\label{thm:CosmoTriangMain:exp}
    \hfill\(
        \EE[K_{n,p}] \simeq n^4p^2\qquad\text{and}\qquad\VV(K_{n,p})
        \gtrsim
        n^6p^3(1-p).
    \)\hfill\mbox{}

    \item\label{thm:CosmoTriangMain:clt}
    \hfill\(\displaystyle
        d_K\!\left(
        \frac{K_{n,p}-\EE[K_{n,p}]}{\sqrt{\VV(K_{n,p})}},
        N\right)
        \lesssim
        \frac{1}{n\sqrt{p(1-p)}},
    \)\hfill\mbox{}\\[0.4\baselineskip]
    where \(N\sim\mathcal N(0,1)\) is a standard Gaussian random variable.
\end{enumerate}
\end{thm}

The strategy of the proof follows the same general scheme as in Section \ref{sec:CosmoEdges}. First, we use the deterministic edge count for unimodular triangulations of cosmological polytopes to obtain an explicit representation of \(K_{n,p}\) in terms of elementary graph statistics. We then compute its expectation, estimate its variance from below, and finally establish suitable first- and second-order discrete gradient bounds. These ingredients are then combined with the normal approximation bound from Proposition \ref{prop:2ndOrderPoincare}. 

The deterministic input is the following lemma.

\begin{lemma}\label{lem:CosmoTriangEdgeFormula}
Let \(G=(V,E)\) be a finite graph, and let \(\Delta\) be a unimodular triangulation of \(\Cosmological_G\). If
$\widetilde n \coloneqq |\{v\in V:\deg(v)\geq 1\}|$ denotes the number of non-isolated nodes of \(G\), then
\[
    f_1(\Delta)
    =
    16\binom{|E|}{2}
    +
    \binom{\widetilde n}{2}
    +
    4\widetilde n\,|E|.
\]
\end{lemma}

\begin{proof}
 Each arc of \(G\) gives rise to four vertices of the triangulation, and any two vertices corresponding to different arcs form an edge. Thus pairs of distinct arcs contribute
\(
    16\binom{|E|}{2}
\)
edges. Next, each arc \(f=\{i,j\}\in E\) contributes the two internal edges
$\{\ee_f,y_{ij}\}$ and  $\{\ee_f,y_{ji}\}$.
This gives \(2|E|\) edges.

It remains to count the edges involving a vertex \(\ee_v\), where \(v\) is a non-isolated node of \(G\). The admissible partners are all vertices of the triangulation except those of the form \(y_{uv}\) with \(uv\in E\). Hence, these edges contribute
\[
    \binom{\widetilde n}{2}
    +
    2\widetilde n\,|E|
    +
    2(\widetilde n-1)|E|.
\]
Here, the first term counts pairs of vertices of the form \(\ee_v,\ee_w\), the second term counts pairs of vertices  of the form \(t_f\), and the last term counts the admissible pairs involving vertices of the form \(y_{ij}\).

Combining these contributions yields
\[
\begin{aligned}
    f_1(\Delta)
    &=
    16\binom{|E|}{2}
    +
    2|E|
    +
    \binom{\widetilde n}{2}
    +
    2\widetilde n\,|E|
    +
    2(\widetilde n-1)|E|  
    =
    16\binom{|E|}{2}
    +
    \binom{\widetilde n}{2}
    +
    4\widetilde n\,|E|,
\end{aligned}
\]
as claimed.
\end{proof}

Applied to \(G_{n,p}\), Lemma \ref{lem:CosmoTriangEdgeFormula} gives the representation
\begin{equation}\label{eq:CosmoTriangKRepresentation}
    K_{n,p}
    =
    16\binom{E_{n,p}}{2}
    +
    \binom{\widetilde N_{n,p}}{2}
    +
    4\widetilde N_{n,p}E_{n,p}.
\end{equation}
Thus, the probabilistic analysis of \(K_{n,p}\) reduces to the joint behaviour of the number of arcs and the number of non-isolated nodes of \(G_{n,p}\).

\subsection{Expectation}\label{sec:CosmoTriangExpectation}

We begin with the expectation of \(K_{n,p}\). 
By
\eqref{eq:CosmoTriangKRepresentation}, the random variable \(K_{n,p}\)
is expressed in terms of the number of arcs \(E_{n,p}\) and the number
\(\widetilde N_{n,p}\) of non-isolated nodes of \(G_{n,p}\). In particular,
\[
    \EE[K_{n,p}]
    =
    16\,\EE\!\left[\binom{E_{n,p}}{2}\right]
    +
    \EE\!\left[\binom{\widetilde N_{n,p}}{2}\right]
    +
    4\,\EE[\widetilde N_{n,p}E_{n,p}].
\]
Thus, it remains
to compute these expectations.

\begin{lemma}\label{lem:CosmoTriangExpectationExact}
Let \(K_{n,p}=f_1(\Delta_{n,p})\) be the number of edges of a unimodular
triangulation of \(\Cosmological_{n,p}\). Then,
\[
\begin{aligned}
\EE[K_{n,p}]
&=
16\binom{\binom{n}{2}}{2}p^2 
+
\binom{n}{2}
\left(
1
-
2(1-p)^{n-1}
+
(1-p)^{2n-3}
\right)  \\
&\qquad
+
4n(n-1)p
+
4n\binom{n-1}{2}p\left(1-(1-p)^{n-1}\right).
\end{aligned}
\]
\end{lemma}

\begin{proof}
Since \(E_{n,p}\sim\mathrm{Bin}\!\left(\binom{n}{2},p\right)\), we have
\[
    \EE\!\left[\binom{E_{n,p}}{2}\right]
    =
    \binom{\binom{n}{2}}{2}p^2.
\]
Next, we compute the expectation of
\(\binom{\widetilde N_{n,p}}{2}\). Distinct nodes \(u,v\in[n]\) are non-isolated unless at least one of them is isolated.
Hence, by the inclusion-exclusion principle,
\[
\PP(u\text{ and }v\text{ are non-isolated})
=
1
-
2(1-p)^{n-1}
+
(1-p)^{2n-3}.
\]
Indeed, the probability that a fixed node is isolated is
\((1-p)^{n-1}\), while the probability that both \(u\) and \(v\) are
isolated is \((1-p)^{2n-3}\), since precisely \(2n-3\) arcs are incident
to at least one of the two nodes. Therefore
\[
    \EE\!\left[\binom{\widetilde N_{n,p}}{2}\right]
    =
    \binom{n}{2}
    \left(
    1
    -
    2(1-p)^{n-1}
    +
    (1-p)^{2n-3}
    \right).
\]

It remains to compute \(\EE[\widetilde N_{n,p}E_{n,p}]\). We write
\[
    \widetilde N_{n,p}
    =
    \sum_{v\in[n]}
    \indicator_{\{\deg(v)\geq 1\}},
    \qquad
    E_{n,p}
    =
    \sum_{e\in\binom{[n]}{2}}
    \indicator_{\{e\in E(G_{n,p})\}},
\]
so that
\[
    \EE[\widetilde N_{n,p}E_{n,p}]
    =
    \sum_{v\in[n]}
    \sum_{e\in\binom{[n]}{2}}
    \PP\bigl(\deg(v)\geq 1,\ e\in E(G_{n,p})\bigr).
\]
We distinguish whether an arc \(e\) is incident to a node  \(v\) or not. If \(e\) is
incident to \(v\), then the event \(e\in E(G_{n,p})\) already implies
\(\deg(v)\geq 1\), and this event occurs with  probability \(p\). There are
\(n(n-1)\) such ordered pairs \((v,e)\). If \(e\) is not incident to
\(v\), then \(e\in E(G_{n,p})\) is independent of the event
\(\{\deg(v)\geq 1\}\). Hence, the corresponding probability is
$p\left(1-(1-p)^{n-1}\right)$.
For each fixed \(v\), there are \(\binom{n-1}{2}\) arcs not incident to
\(v\). Consequently,
\[
    \EE[\widetilde N_{n,p}E_{n,p}]
    =
    n(n-1)p
    +
    n\binom{n-1}{2}p\left(1-(1-p)^{n-1}\right).
\]
Combining the three identities proves the lemma.
\end{proof}

\begin{proof}[Proof of Theorem \ref{thm:CosmoTriangMain} (a), expectation]
Recall the exact formula for $\EE[K_{n,p}]$ from Lemma \ref{lem:CosmoTriangExpectationExact}. The first term is of order \(n^4p^2\). We now show that the remaining
terms are bounded by terms of the same order. Using
\(1-(1-p)^{n-1}\leq (n-1)p\), we obtain
\[
    n\binom{n-1}{2}p\left(1-(1-p)^{n-1}\right)
    \lesssim n^4p^2.
\]
Moreover, for the second term in the exact formula, we use the identity
\[
\begin{aligned}
1
-
2(1-p)^{n-1}
+
(1-p)^{2n-3}
&=
\left(1-(1-p)^{n-1}\right)^2
+
p(1-p)^{2n-3}.
\end{aligned}
\]
Hence
\[
\begin{aligned}
    \binom{n}{2}
    \left(
    1
    -
    2(1-p)^{n-1}
    +
    (1-p)^{2n-3}
    \right)
    &\leq
    \binom{n}{2}
    \left(n^2p^2+p\right)
    \lesssim
    n^4p^2+n^2p
    \lesssim
    n^4p^2,
\end{aligned}
\]
where the last step uses \(p\gg n^{-2}\). Finally,
\(n(n-1)p\lesssim n^4p^2\) also follows from \(p\gg n^{-2}\). Thus all
remaining terms are bounded by a constant multiple of \(n^4p^2\), while
the first term is itself of order \(n^4p^2\). This proves $\EE[K_{n,p}]\simeq n^4p^2$
whenever \(p\gg n^{-2}\).
\end{proof}

\subsection{Variance}\label{sec:CosmoTriangVariance}

We next turn to the variance of \(K_{n,p}\) and to the proof of the second claim in Theorem \ref{thm:CosmoEdgesMain} (a). In contrast to the expectation,
we do not derive an exact formula. Instead, we develop a lower bound of the correct order, which is enough for the proof of the quantitative central limit theorem. 

The representation \eqref{eq:CosmoTriangKRepresentation} shows that
\(K_{n,p}\) contains three different contributions: a quadratic term in the
number of arcs, a quadratic term in the number of non-isolated nodes, and a
mixed term. The variance estimate below reflects these sources of
fluctuations. In this subsection and the next one
we write $q\coloneqq 1-p$.

\begin{lemma}\label{lem:CosmoTriangVarianceLowerBound}
Let \(K_{n,p}=f_1(\Delta_{n,p})\) be the number of edges of a unimodular
triangulation of \(\Cosmological_{n,p}\). Then,
\[
    \VV(K_{n,p})
    \gtrsim
    n^6p^3q
    +
    n^5p^2q^{n-1}(1-q^{n-1})
    +
    n^3pq(1-q^{n-1}).
\]
In particular, if \(p\gg n^{-2}\), then 
\[
    \VV(K_{n,p})
    \gtrsim
    n^6p^3q .
\]
\end{lemma}

The first term in Lemma~\ref{lem:CosmoTriangVarianceLowerBound} reflects the
fluctuation coming from the quadratic arc-count contribution
\(\binom{E_{n,p}}{2}\). The remaining two terms come from the mixed
contribution involving the number of non-isolated nodes. The quadratic term
\(\binom{\widetilde N_{n,p}}{2}\) is present in the representation of
\(K_{n,p}\), but it is not needed for the lower bound and will be discarded. The factor
\(1-q^{n-1}\) is the probability that a fixed node is non-isolated.

\begin{proof}[Proof of Lemma~\ref{lem:CosmoTriangVarianceLowerBound}]
Defining
\[
    X_{n,p}\coloneqq \binom{E_{n,p}}{2},
    \qquad
    Y_{n,p}\coloneqq \binom{\widetilde N_{n,p}}{2},
    \qquad
    Z_{n,p}\coloneqq \widetilde N_{n,p}E_{n,p},
\]
we have
\[
    K_{n,p}=16X_{n,p}+Y_{n,p}+4Z_{n,p}
\]
by Lemma \ref{lem:CosmoTriangEdgeFormula}.
All three random variables are increasing functions of the arc indicators.
Hence, by the FKG inequality for product Bernoulli measures, all covariance
terms between them are non-negative, see \cite[Theorem 4.11]{GrimmettProbabilityOnGraphs} or \cite[Problem 3.11.18]{GrimmettStirzaker1000Exercises}. Consequently,
\[
    \VV(K_{n,p})\gtrsim \VV(X_{n,p})+ \VV(Y_{n,p})+ \VV(Z_{n,p}) \gtrsim \VV(X_{n,p})+\VV(Z_{n,p}).
\]

We first consider \(X_{n,p}\). This random variable counts unordered pairs
of present arcs. If two such pairs have exactly one arc in common, the
covariance of their indicators is \(p^3q\). Since the number of such pairs
is of order \(n^6\), we obtain $\VV(X_{n,p})\gtrsim n^6p^3q$.

It remains to extract the additional contribution coming from \(Z_{n,p}\).
We write
\[
    Z_{n,p}
    =
    \sum_{v\in[n]}
    \sum_{e\in\binom{[n]}{2}}
    \indicator_{\{e\in E(G_{n,p})\}}
    \indicator_{\{\deg(v)\geq 1\}} .
\]
Again, all summands are increasing functions of the arc indicators, so all
covariances in this double sum are non-negative by the FKG inequality. We may therefore keep only
selected classes of terms. First, let us fix a node \(v\) and two distinct arcs \(e,f\) which are not incident
to \(v\). There are of order \(n^5\) such triples. The arc indicators of
\(e\) and \(f\) are independent of the event \(\{\deg(v)\geq 1\}\). Thus,
\[
\begin{aligned}
&\Cov\!\left(
    \indicator_{\{e\in E(G_{n,p})\}}\indicator_{\{\deg(v)\geq 1\}},
    \indicator_{\{f\in E(G_{n,p})\}}\indicator_{\{\deg(v)\geq 1\}}
\right)  
=
p^2q^{n-1}(1-q^{n-1}).
\end{aligned}
\]
This gives
\[
    \VV(Z_{n,p})
    \gtrsim
    n^5p^2q^{n-1}(1-q^{n-1}).
\]
Second, we keep the diagonal terms with \(v\notin e\). There are of order
\(n^3\) such pairs \((v,e)\). For each of them, the product $\indicator_{\{e\in E(G_{n,p})\}}\indicator_{\{\deg(v)\geq 1\}}$
is a Bernoulli random variable with parameter \(p(1-q^{n-1})\). Its variance
is therefore bounded from below by
\[
    p(1-q^{n-1})\bigl(1-p(1-q^{n-1})\bigr)
    \geq
    pq(1-q^{n-1}).
\]
Hence $\VV(Z_{n,p})\gtrsim n^3pq(1-q^{n-1})$.
Combining these estimates proves
\[
    \VV(K_{n,p})
    \gtrsim
    n^6p^3q
    +
    n^5p^2q^{n-1}(1-q^{n-1})
    +
    n^3pq(1-q^{n-1}),
\]
as claimed.

It remains to justify the final assertion. The first term is already
\(n^6p^3q\). For the second term, using
\(1-q^{n-1}\leq (n-1)p\), we get
\[
    n^5p^2q^{n-1}(1-q^{n-1})
    \leq
    n^6p^3q^{n-1}
    \lesssim
    n^6p^3q .
\]
For the third term, we use
$1-q^{n-1}\leq \min\{1,np\}$.
If \(p\leq n^{-1}\), then
\[
    n^3pq(1-q^{n-1})
    \leq
    n^4p^2q
    =
    \frac{1}{n^2p}\,n^6p^3q
    \lesssim
    n^6p^3q,
\]
where the last estimate follows from \(p\gg n^{-2}\). If, on the other hand, \(p>n^{-1}\), then
\[
    n^3pq(1-q^{n-1})
    \leq
    n^3pq
    \leq
    n^6p^3q .
\]
Thus, for \(p\gg n^{-2}\), the right-hand side in the variance bound is of
order \(n^6p^3q\), and it follows that
\[
    \VV(K_{n,p})\gtrsim n^6p^3q .
\]
This completes the proof.
\end{proof}

\subsection{Discrete gradients and a quantitative central limit theorem}
\label{sec:CosmoTriangCLT}

We now estimate the discrete gradients of \(K_{n,p}\) and prove the
quantitative central limit theorem. As in Section~\ref{sec:prelim_prob}, we
encode the random graph \(G_{n,p}\) by independent Rademacher variables
indexed by the possible arcs \(e\in\binom{[n]}{2}\). 
For an arc \(e\in\binom{[n]}{2}\), the operator \(\D_e\) denotes the
corresponding discrete gradient. We also write $q\coloneqq 1-p$ again. 

The representation \eqref{eq:CosmoTriangKRepresentation}
shows that the only new ingredient, compared with Section~\ref{sec:CosmoCLT},
is the number of non-isolated nodes. We first record the relevant bounds for the discrete gradients.

\begin{lemma}\label{lem:CosmoTriangFirstGradient}
For every \(e\in\binom{[n]}{2}\),
\[
    \EE\bigl[(\D_e K_{n,p})^4\bigr]
    \lesssim
    n^8p^6q^2 .
\]
\end{lemma}

\begin{proof}
Fix \(e=\{u,v\}\in\binom{[n]}{2}\). Changing the state of \(e\) changes
\(E_{n,p}\) by exactly one. Hence $|\D_e E_{n,p}|=\sqrt{pq}$.
Moreover,
\[
    \left|\D_e\binom{E_{n,p}}{2}\right|
    \leq
    \sqrt{pq}\,E_{n,p}.
\]

For the other terms, we use the following product identity from \cite[Proposition 7.8]{Privault2008}, which says that
\begin{equation}\label{eq:ProductEstimate} 
    \D_e(AB)
    =
    (\D_e A)B_e
    +
    A_e(\D_e B)
    +
    \frac{1}{\sqrt{pq}}(\D_e A)(\D_e B)
\end{equation}
for any two functionals $A$ and $B$ of the arc variables.

We next consider \(\widetilde N_{n,p}\). Changing the state of \(e\) can
only affect the isolation status of the two endpoints \(u\) and \(v\). Thus
$|\D_e\widetilde N_{n,p}|\leq 2\sqrt{pq}$. Consequently, using \eqref{eq:ProductEstimate} we obtain
\begin{align}\label{eq:DeNonIsolated}
    \left|\D_e\binom{\widetilde N_{n,p}}{2}\right| &=  \frac12 \left |
    (\D_e\widetilde N_{n,p})(\widetilde N_{n,p} -1)
    +
    \widetilde N_{n,p}(\D_e\widetilde N_{n,p}-1) 
    +
    \frac{1}{\sqrt{pq}}(\D_e\widetilde N_{n,p})^2\right | \\ 
    &\lesssim \sqrt{pq}(\widetilde N_{n,p}+1).\nonumber
\end{align}

 
For the mixed term using the triangle inequality and \eqref{eq:ProductEstimate}  with
\(A=\widetilde N_{n,p}\) and \(B=E_{n,p}\)  yields
\[
\begin{aligned}
    |\D_e(\widetilde N_{n,p}E_{n,p})|
    &\lesssim
    |\D_e\widetilde N_{n,p}|\,E_{n,p}
    +
    \widetilde N_{n,p}|\D_eE_{n,p}|
    +
    \frac{1}{\sqrt{pq}}
    |\D_e\widetilde N_{n,p}|\,|\D_eE_{n,p}|  \\
    &\lesssim
    \sqrt{pq}\,(E_{n,p}+\widetilde N_{n,p}+1).
\end{aligned}
\]
Since every non-isolated node is incident to at least one present arc, we
have \(\widetilde N_{n,p}\leq 2E_{n,p}\). Combining the preceding estimates
therefore yields
\[
    |\D_eK_{n,p}|
    \lesssim
    \sqrt{pq}\,(E_{n,p}+1).
\]
It follows that
\[
    (\D_eK_{n,p})^4
    \lesssim
    p^2q^2\,(E_{n,p}^4+1).
\]
Since $E_{n,p}\sim \mathrm{Bin}\!\left(\binom{n}{2},p\right)$,
we have $\EE[E_{n,p}^4]\lesssim n^8p^4$
in the regime considered here. Hence,
\[
    \EE\bigl[(\D_eK_{n,p})^4\bigr]
    \lesssim
    p^2q^2\, n^8p^4
    =
    n^8p^6q^2,
\]
as claimed.
\end{proof}

The second-order gradients have a mild dependence on whether the two arcs
share an endpoint.

\begin{lemma}\label{lem:CosmoTriangSecondGradient}
For all \(e,f\in\binom{[n]}{2}\) with $e\neq f$,
\[
    \EE\bigl[(\D_f\D_eK_{n,p})^4\bigr]
    \lesssim
    \begin{cases}
        p^4q^4,
        &\text{if } e\cap f=\emptyset,\\[0.2cm]
        p^4q^4+n^8p^8q^{\,n+1},
        &\text{if } |e\cap f|=1.
    \end{cases}
\]
\end{lemma}
\begin{proof}
The contribution coming from \(\binom{E_{n,p}}{2}\) is the same as in
Section~\ref{sec:CosmoCLT}. Since changing one arc changes \(E_{n,p}\) by
one,
\[
    \left|\D_f\D_e\binom{E_{n,p}}{2}\right|
    \leq pq .
\]

Next, we consider \(\widetilde N_{n,p}\).
As before, changing the state of an arc can
only affect the isolation status of its two endpoints. 
Thus
\(|\D_e\widetilde N_{n,p}|\leq 2\sqrt{pq}\). 
If \(e\cap f=\emptyset\), then
the two changes concern disjoint sets of endpoints, and
\(\D_f\D_e\widetilde N_{n,p}=0\).
If \(e\cap f=\{x\}\), then a second-order
contribution can only occur when the common endpoint \(x\) becomes isolated
after the arcs \(e\) and \(f\) have been removed.
We write
\(\deg_{G-e-f}(x)\) for the degree of \(x\) in the graph obtained from
\(G_{n,p}\) by deleting \(e\) and \(f\). Then
\[
    |\D_f\D_e\widetilde N_{n,p}|
    \lesssim
    pq\,\indicator_{\{\deg_{G-e-f}(x)=0\}} .
\]

We now apply the discrete gradient \(\D_f\) to the three terms on the right-hand side of \eqref{eq:DeNonIsolated} and use the estimate
\eqref{eq:ProductEstimate} for each product. Since $|\D_e\widetilde N_{n,p}|\lesssim \sqrt{pq}$, $|\D_f\widetilde N_{n,p}|\lesssim \sqrt{pq}$ and
$|\D_f\D_e\widetilde N_{n,p}|\lesssim pq\,\indicator_{\{\deg_{G-e-f}(x)=0\}}$
in the intersecting case, while it is zero in the disjoint case, we obtain
\[
    \left|\D_f\D_e\binom{\widetilde N_{n,p}}{2}\right|
    \lesssim
    pq
    +
    \widetilde N_{n,p}\,|\D_f\D_e\widetilde N_{n,p}|.
\]
Therefore
\[
    \left|\D_f\D_e\binom{\widetilde N_{n,p}}{2}\right|
    \lesssim
    \begin{cases}
        pq, &\text{if } e\cap f=\emptyset,\\[0.15cm]
        pq\,(1+\widetilde N_{n,p}\indicator_{\{\deg_{G-e-f}(x)=0\}}),
        &\text{if } e\cap f=\{x\}.
    \end{cases}
\]
Similarly, for the mixed term
\(\widetilde N_{n,p}E_{n,p}\) we obtain
\[
    |\D_f\D_e(\widetilde N_{n,p}E_{n,p})|
    \lesssim
    \begin{cases}
        pq, &\text{if } e\cap f=\emptyset,\\[0.15cm]
        pq\,(1+E_{n,p}\indicator_{\{\deg_{G-e-f}(x)=0\}}),
        &\text{if } e\cap f=\{x\}.
    \end{cases}
\]
Since \(\widetilde N_{n,p}\leq 2E_{n,p}\), the three contributions in
\eqref{eq:CosmoTriangKRepresentation} imply
\[
    |\D_f\D_eK_{n,p}|
    \lesssim
    \begin{cases}
        pq, &\text{if } e\cap f=\emptyset,\\[0.15cm]
        pq\,(1+E_{n,p}\indicator_{\{\deg_{G-e-f}(x)=0\}}),
        &\text{if } e\cap f=\{x\}.
    \end{cases}
\]
Taking fourth powers gives the asserted bound in the disjoint case. In the
intersecting case, it remains to control the exceptional event. On
\(\{\deg_{G-e-f}(x)=0\}\), all arcs incident to \(x\), except possibly
\(e\) and \(f\), are absent. The remaining arcs are independent of this
event and are distributed as an Erd\H{o}s--R\'enyi graph on \(n-1\) nodes,
up to the two fixed arcs. Hence
\[
    \EE\!\left[
        \indicator_{\{\deg_{G-e-f}(x)=0\}}E_{n,p}^4
    \right]
    \lesssim
    q^{n-3}n^8p^4 .
\]
Consequently, whenever \(e\cap f=\{x\}\),
\[
    \EE\bigl[(\D_f\D_eK_{n,p})^4\bigr]
    \lesssim
    p^4q^4+n^8p^8q^{n+1}.
\]
This proves the lemma.
\end{proof}

We are now ready to prove the normal approximation bound.

\begin{proof}[Proof of Theorem~\ref{thm:CosmoTriangMain}~\ref{thm:CosmoTriangMain:clt}]
Set
\[
    F_{n,p}
    \coloneqq
    \frac{K_{n,p}-\EE[K_{n,p}]}{\sqrt{\VV(K_{n,p})}}.
\]
We apply Proposition~\ref{prop:2ndOrderPoincare} to \(F_{n,p}\). By
Lemma~\ref{lem:CosmoTriangVarianceLowerBound},
\[
    \VV(K_{n,p})
    \gtrsim
    n^6p^3q+n^5p^2q^{n-1}+n^3pq^{n-1}.
\]
Moreover, by Lemmas~\ref{lem:CosmoTriangFirstGradient} and
\ref{lem:CosmoTriangSecondGradient}, and by Hölder's inequality,
\[
    \EE\bigl[(\D_eF_{n,p})^4\bigr]
    \lesssim
    \frac{n^8p^6q^2}{\VV(K_{n,p})^2}\qquad\text{and}\qquad
    \EE\bigl[(\D_f\D_eF_{n,p})^4\bigr]
    \lesssim
    \frac{p^4q^4+\indicator_{\{|e\cap f|=1\}}n^8p^8q^{n+1}}
    {\VV(K_{n,p})^2}
\]
for all \(e,f\in\binom{[n]}{2}\).
Substitution into the quantities \(B_1,\ldots,B_5\) from
Proposition~\ref{prop:2ndOrderPoincare}, with the sums split according to
whether the corresponding arcs are disjoint or share one endpoint, gives
\[
\begin{aligned}
B_1
&\lesssim
\frac{
    n^{10}p^5q^3
    + n^{11}p^6q^{(n+1)/4+2}
    + n^{12}p^7q^{(n+1)/2+1}
}{\VV(K_{n,p})^2},                                      \\
B_2
&\lesssim
\frac{
    n^6p^3q^3
    + n^9p^5q^{(n+1)/2+1}
    + n^{12}p^7q^n
}{\VV(K_{n,p})^2},                                      \\
B_3
&\lesssim
\frac{n^{10}p^5q}{\VV(K_{n,p})^2},\qquad
B_4
\lesssim
\frac{
    n^8p^4q^2
    + n^{11}p^6q^{(n+1)/2}
}{\VV(K_{n,p})^2},                                      \\
B_5
&\lesssim
\frac{
    n^4p^2q^2
    + n^{11}p^6q^{n-1}
}{\VV(K_{n,p})^2}.
\end{aligned}
\]
The additional terms involving powers of \(q^n\) do not dominate the
corresponding contribution in \(B_3\). Indeed, if \(p\gg n^{-1}\), these
terms are exponentially small. If \(p\lesssim n^{-1}\), then \(q^n\) is bounded from below by a positive
constant, and the additional polynomial factors are still dominated by
\(n^{10}p^5q\) in the range \(p\gg n^{-2}\). Thus
\[
    B_1+\ldots+B_5
    \lesssim
    \frac{n^{10}p^5q}{\VV(K_{n,p})^2} \lesssim
    \frac{1}{n^2pq}.
\]
Proposition~\ref{prop:2ndOrderPoincare} therefore yields
\[
    d_K(F_{n,p},N)
    \lesssim
    \frac{1}{n\sqrt{pq}}
    =
    \frac{1}{n\sqrt{p(1-p)}},
\]
and this proves Theorem~\ref{thm:CosmoTriangMain}~\ref{thm:CosmoTriangMain:clt}.
\end{proof}

\subsubsection*{Acknowledgement}
This work has been supported by the German Research Foundation (DFG) via SPP 2458 \textit{Combinatorial Synergies} and via project number 547295909.

\bibliographystyle{plain}  
\bibliography{quellen}   

\end{document}